\documentclass{amsart}
\usepackage[a4paper, margin=1in]{geometry} 
\usepackage[fontsize=12bp]{fontsize}

\usepackage{amsthm}

\newtheorem{theorem}{Theorem}[section]

\newtheorem{lemma}[theorem]{Lemma}
\newtheorem{corollary}[theorem]{Corollary}

\theoremstyle{definition}
\newtheorem{definition}[theorem]{Definition}
\newtheorem{rem}[theorem]{Remark}
\newtheorem{example}[theorem]{Example}
\numberwithin{equation}{section}

\usepackage{amsfonts} 

\newcommand{\Z}{\mathbb{Z}}
\newcommand{\tit}{\textit}

\usepackage{graphicx} 

\usepackage{amssymb}  

\usepackage[nobysame, alphabetic]{amsrefs}
\usepackage{url}
\usepackage[]{hyperref}
\hypersetup{
    colorlinks = true,
    citecolor = black
}
\usepackage[all]{hypcap} 

\graphicspath{ {./} }

\begin{document}

\title{ Virtual Cactus Group is Virtually Special Compact}
\author{Maninder S. Dhanauta}
\date{}

\begin{abstract}
    We prove that the virtual cactus group has a finite index subgroup that is the fundamental group of a compact special cube complex.
\end{abstract}

\maketitle

\section{Introduction}


\subsection{The cactus group and the virtual cactus group}

The pure cactus group  is the fundamental group of the Deligne-Mumford compactification of the moduli space of genus 0 real curves with $n+1$ marked points. Devadoss and Kapranov tiled the Deligne-Mumford space by associahedra~\cite{Dev99}~\cite{Kap93}. Davis--Januszkiewicz--Scott showed that the structure dual to the tiling was a CAT(0) cube complex and studied a more general class of ``mock reflection groups'' which acted on cubulated manifolds~\cite{DJS03}. A more succinct and general definition was provided by Scott. A \tit{right-angled mock reflection group} is any group acting isometrically on a CAT(0) cube complex such that the action is simply-transitive on the 0-skeleton, and the stabilizer of every edge is isomorphic to
 $\Z/2\Z$~\cite{Sco08}. In recent years, Genevois used tools from geometric group theory to study the cactus group ~\cite{Gen22}.

Denote $[n] = \{ 1, \dots, n\}$ and $[i,j] = \{i,i+1,\dots, j\}$ where $i<j$ where $i,j \in \mathbb{N}$. The \tit{cactus group}~$C_n$ is generated by $\{ s_{ij} \mid 1\leq i < j \leq n \}$ and has relations
    \begin{enumerate}
        \itemsep0em 
        \item $s_{ij}^2 = 1$ 
        \item $s_{ij}s_{kl} = s_{kl}s_{ij}$ if $[i,j] \cap [k,l] = \varnothing$
        \item $s_{ij}s_{kl} = s_{w_{ij}(l)w_{ij}(k)}s_{ij}$ if $ [k,l] \subset [i,j]$
    \end{enumerate}
    for all $i<j, k<l$ where $w_{ij} $ is the element of the symmetric group $S_n$ which reverses the interval $[i,j]$ and acts as the identity on the elements outside of the interval. 

    The \tit{virtual cactus group} $vC_n$ is the quotient of the free product $C_n*S_n$ by the relations
    $$ws_{ij}w^{-1} = s_{w(i)w(j)}$$    
    for all $1\leq i < j \leq n$ and $w\in S_n$ such that $w(i+k) = w(i)+k$ for $k=0,\dots,j-i$. This group is introduced in~\cite{IKLPR23}*{Section 10.5} and contains a copy of the cactus group.

Define a homomorphism $C_n \to S_n$ that maps generators $s_{ij}$ to $w_{ij}$. There is a homomorphism $vC_n \to S_n$ which extends the previous homomorphism and is the identity on the elements of $S_n$. To see that this is indeed a homomorphism, one can verify that the relation $ww_{ij}w^{-1} = w_{w(i)w(j)}$ holds in $S_n$. 
The \tit{pure cactus group} and the \tit{pure virtual cactus group}, denoted $PC_n$ and $PvC_n$ respectively, are the kernels of their respective maps. These are finite index subgroups.

Genevois studied the Cayley graph of the cactus group $C_n$, which is the 1-skeleton of a CAT(0) cube complex ~$\widetilde{D}_n$, and showed that the action of $PC_n$ on $\widetilde{D}_n$ is special~\cite{Gen22}*{Theorem 2.7, Theorem 4.2}. In view of~\cite{HW08}*{Lemma~9.12}, the quotient $D_n$ of $\widetilde{D_n}$ under this action is a special cube complex. In particular, the cactus group is virtually special compact. We will provide all relevant definitions for cube complexes in  Section~\ref{sec:definitions}.

The pure virtual cactus group $PvC_n$ is the fundamental group of a non-positively curved~(NPC) cube complex~$\widehat{D}_n$~\cite{IKLPR23}*{Section~8}. The main results of the paper are as follows.

\begin{theorem}
    The cube complex $\widehat{D}_n$ has a finite degree cover, denoted $M_n$, which is a compact, special cube complex. 
    \label{thm:Mn-special}
\end{theorem}

\begin{corollary}
    The virtual cactus group is virtually special compact.
    \label{cor:vcn-vs}
\end{corollary}

In our construction of the cover, we take two copies of $\widehat{D}_n$,  cut along any hyperplane and glue the two copies along the cuts to get a double cover. Iterating this process for each type of hyperplane, the resulting cover has degree $2^{|L|}$ where $L \subset \mathcal{P}(\{1,\dots,n\})$ contains all non-empty and non-singleton subsets. We show that the hyperplanes in $\widehat{D}_n$ are embedded. As a concequence, we get that $M_n$ has two-sided hyperplanes. We then prove that $M_n$ is a special cube complex. 

From \cite{HW08}, we get the following corollary as a consequences of being virtually special. 

\begin{corollary}
  The virtual cactus group is $\mathbb{Z}$-linear.
\end{corollary}


\section{Definitions}
\label{sec:definitions}

\subsection{Cube complexes, hyperplanes, and specialness}

Special cube complexes and special groups were introduced in~\cite{HW08}. Coxeter groups are examples of virtually special groups~\cite{HW10} while the cactus group is an example of a  virtually special compact group~\cite{Gen22}. We will recall some of the relevant definitions. Refer to ~\cite{Sag14} for a survey of the field.

Consider a collection of cubes $[-1,1]^n$ of various dimensions, each given the euclidean metric. We may glue cubes in this collection along their faces via isometries. The resulting space~$X$ is a \tit{cube complex}.  
Graphs are examples of cube complexes. An $n$-torus, defined as the product $(S^1)^n$, is an example of a cube complex where the opposite faces of an $n$-dimensional cube are glued by translations. 

Consider an $n$-cube $[-1,1]^n$ in $X$. The set of all points in the cube where one coordinate is fixed to be 0 is called a \tit{midcube}. Consider the disjoint union of all midcubes in $X$ and glue them according to the gluings in $X$. The resulting space naturally maps into~$X$. An \tit{(immersed) hyperplane} is the image of a connected component of this resulting space. Note that the restriction of this map to a connected component is not necessarily an embedding. 

If the midpoint of a 1-cube is contained in a midcube, the 1-cube is said to be \tit{dual} to that midcube. The 1-cube is also \tit{dual} to the hyperplane that contains the midcube.

The \tit{link} of a 0-cube $v\in X$, denoted $lk(v)$, is the complex with vertex set being all directed 1-cubes originating from $v$. Any set of $n$ vertices form an $n$-simplex in $lk(v)$ if the corresponding directed 1-cubes span the corner of an $n$-cube in $X$.

A simplicial complex is \tit{flag} if every $n$-clique is contained in an $n$-simplex. A cube complex is \tit{nonpositively curved(NPC)} if the link of every vertex is a flag simplicial complex. Lastly, a CAT(0) \tit{cube complex} is a simply connected, NPC cube complex.


Next, we identify some hyperplane pathologies that may occur in NPC cube complexes.

Suppose a 1-cube $\epsilon_1$ is dual to some hyperplane and is given some orientation. If~$\epsilon_1$ lies in a 2-cube, we orient the 1-cube $\epsilon_2$ that lies on the opposite side in the same direction as~$\epsilon_1$. In this manner, we can extend the orientation to all 1-cubes dual to the hyperplane. If each 1-cube receives a unique orientation, the hyperplane is said to be \tit{two-sided}. We denote an oriented 1-cube as $\vec{\epsilon}$.  

Two hyperplanes $H_1, H_2$ \tit{intersect} if they have dual 1-cubes $\epsilon_1 \neq \epsilon_2$, respectively, that meet at a 0-cube and form a corner of a 2-cube.
Two hyperplanes $H_1, H_2$ \tit{osculate} if they have two oriented dual 1-cubes $\vec{\epsilon_1} \neq \vec{\epsilon_2}$ that originate from the same 0-cube but do not form a corner of a 2-cube.
Two hyperplanes \tit{inter-osculate} if they intersect and osculate.
A hyperplane \tit{self-intersects} if it intersects itself. Equivalently, a hyperplane is \tit{embedded} (does not self-intersect) if the natural map from the quotient space of the midcubes is an injection from the appropriate connected component into the hyperplane. A hyperplane \tit{self-osculates} if it osculates with itself.

With these properties in mind, a cube complex is \tit{special} if it is NPC and satisfies the following conditions,
\begin{itemize}
    \item hyperplanes are two-sided,
    \item hyperplanes do not self-intersect,
    \item hyperplanes do not self-osculate,
    \item hyperplanes do not inter-osculate.
\end{itemize}

A cube complex is \tit{virtually special} if it has a finite degree cover that is a special cube complex. The fundamental group of a special cube complex is said to be \tit{special}. A group~$G$ is said to be \tit{virtually special}  if it has a finite index subgroup that is special. Lastly, a group is  \tit{virtually special compact} if it has a finite index subgroup that is the fundamental group of a compact special cube complex.


\subsection{Planar trees and forests}

We recall the construction of the family of NPC cube complexes,  $\widehat{D}_n$,  introduced in~\cite{IKLPR23}*{Section 8}. 

A \tit{tree} is a connect graph without cycles such that no vertex is contained is exactly 2 edges. A \tit{leaf} is a vertex in a tree that is contained in a single edge. A \tit{rooted} tree is a tree with an identified leaf called the \tit{root}. 
Let edge $e$ contain a vertex $v$. If every path from the root to $v$ contains $e$, we say $e$ is the \tit{descending edge} at $v$. Otherwise $e$ is \tit{ascending at}  $v$. Vertex $v_2$ lies \tit{above} $v_1$ if all paths from $v_2$ to the root pass through $v_1$. Similarly, vertex $v$ lies \tit{above} edge $e$ if all paths from $v$ to the root contain $e$.
 A \tit{planar} tree is a rooted tree with an ordering on the set of ascending edges at each vertex. In other words, it is a rooted tree along with a choice of embedding into the plane (up to isotopy). 
 In a diagram of a planar tree, the order given to the ascending edges at a vertex determines the order in which they appear when going around the vertex in a clockwise manner beginning at the descending edge.
  \textbf{From now on, when we talk about leaves, we exclude the root.} 
 
 A \tit{planar forest labelled by $[n]$} is a sequence of planar trees that have $n$ leaves in total amoung all the trees. Furthermore, each leaf receives a unique label from the set $[n] =\{1,\dots,n\}$.  
An \tit{internal edge} of a planar forest is an edge that is disjoint from the leaves. Equivalently, it is an edge that has at least two leaves above it. 
The set of all internal edges of a planar forest $\tau$ is denoted by $E(\tau)$. 
For all $e \in E(\tau)$, we can visit all the leaves that lie above $e$ by a depth-first search that respects the order of ascending edges at each vertex. As we visit each leaf, the order in which we read the label of each leaf is denoted as $\mathcal{O}(\tau, e)$ and is an ordered subset of $[n]$. If $\tau$ has one internal edge, we denote $\mathcal{O}(\tau)$ to mean the ordered labels above that edge. If $A=(a_1,\dots,a_m)$ is an ordered subset of $[n]$, define $A^r = (a_m,\dots,a_1)$. If $A,B,C$ are disjoint ordered subsets of~$[n]$, denote $ABC$ as the concatenation.

The set of all planar forests labelled by $[n]$ with $k$ internal edges is denoted $PF_n(k)$. Later, these will label $k$-cubes in our cube complex. 

We can define a flipping operation at an internal edge $e \in E(\tau)$ of a forest $\tau$.  
Suppose the edge $e$ is ascending at $v_1$ and descending at $v_2$. We define the forest $r_e{(\tau)}$ to be the result of reversing the order of all ascending edges at each vertex $v'$ equal to or above $v_2$. 

Using the same setup, the deleting operation $d_e(\tau)$ collapses the internal edge $e$ and inserts the ordered set of ascending edges at $v_2$ into the set at $v_1$ in the obvious way such that the order in which a depth-first search would visit the leaves remains unchanged. However, if $v_1$ is a root of a planar tree $\tau_i$ in the planar forest $\tau = (\tau_1,\dots, \tau_m)$, we consider the rooted, planar forest $(\tau_1 ', \dots, \tau_p ')$  we get by deleting the closed edge $e$ from $\tau_i$ and giving each new tree a distinct root. In this case, we define $d_e(\tau) = ( \tau_1,\dots, \tau_{i-1}, \tau_1 ', \dots, \tau_p ' , \tau_{i+1}, \dots, \tau_m)$.


Let $PF_n =  \displaystyle \bigcup_{k=0}^n PF_n(k)$. Now to define $D_n$,
$$D_n = \bigcup_{\tau \in PF_n} \{\tau \} \times [-1,1]^{E(\tau)}/\sim$$
where
$$(\tau, (t,s)) \sim (r_e(\tau), (t,-s)) \verb|  | \mathrm{ and } \verb|  |
(\tau, (t,1)) \sim (d_e(\tau), t)$$

Here $e \in E(\tau)$, and tuple $(t,s) \in [-1,1]^{E(\tau) - \{e\}} \times [-1,1]^{\{e\}}$ has the value $s$ in the coordinate corresponding to the edge $e$. If $E(\tau)$ is empty, $[-1,1]^{E(\tau)}$ is a single point.

By removing the midcubes of each cube $[-1,1]^{E(\tau)}$, we get $2^k$ connected components. The closure of each component is a cube, the \tit{subcubes} are the collection of these cubes. We call $[0,1]^{E(\tau)}$ the \tit{positive subcube}. Each subcube of $[-1,1]^{E(\tau)}$ is identified with the positive subcube of some unique forest $\tau'$ obtained from $\tau$ by a sequence of flippings. See 

We can also construct $D_n$ by gluing together only the positive subcubes of each cube,

$$  D_n = \bigcup_{\tau \in PF_n} \{\tau \} \times [0,1]^{E(\tau)}/\sim  $$
where
$$(\tau, (t,0)) \sim (r_e(\tau), (t,0)) \verb|  |  \mathrm{and} \verb|  | 
(\tau, (t,1)) \sim (d_e(\tau), t)$$

Given two opposite faces of a subcube, one face is contained in a midcube of a cube in~$D_n$. We call such a face being \tit{inward}. It is glued directly to the face of  another subcube of the same dimension. The other \tit{outward} face is contained in a face of a cube in $D_n$ and is glued to a subcube of 1 dimension less. See Figure~\ref{fig:loops-Dn} for an example.

Lastly, we identify the cube $(\tau, [0,1]^{E(\tau)} )$ and $(\tau', [0,1]^{E(\tau')} )$ if $\tau=(\tau_1, \dots, \tau_m)$ and $\tau'$ is some permutation of the sequence $\tau$. We call this quotient cube complex $\widehat{D}_n$. 

\begin{example}
    The highest dimensional cubes in $D_3$ are 2-dimensional. The labels for these cubes are forests with a single tree with two internal edges. 
    As seen in the Figure~\ref{fig:example_of_Dn}, there are only three 2-cubes.
     There are twelve 2-subcubes.
\end{example}

\begin{figure}[h]
    \centering
    \includegraphics[width=0.7\textwidth]{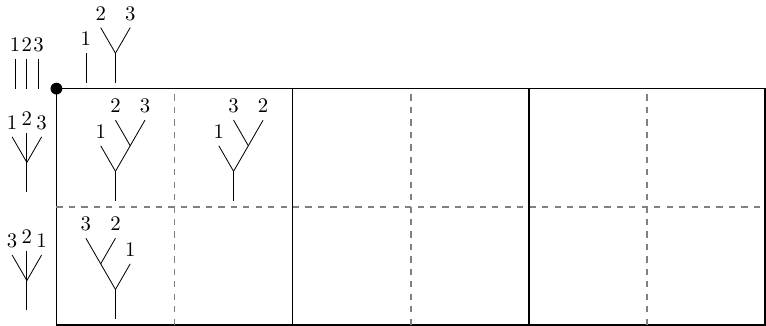}      
    \caption{A partial sketch of the cube complex $D_3$ using subcubes.}
    \label{fig:example_of_Dn}
\end{figure}

Generally speaking, suppose that a midcube $Z$ is contained a 2-cube which also contains an oriented 1-cube $\vec{\epsilon}_1$. We define $Z \cdot \vec{\epsilon}_1$ to be the reflection of $\vec{\epsilon}_1$ about $Z$. 
Consider a sequence of oriented 1-cubes $\vec{\epsilon}_i = Z_{i-1} \cdot \vec{\epsilon}_{i-1}$ beginning at $\vec{\epsilon_1}$ and where each pair $Z_i, \vec{\epsilon}_i$ is contained in some 2-cube. The 1-cube $\vec{\epsilon}_i$ is either the reverse of $\vec{\epsilon}_{i-1}$ or is the 1-cube with the same orientation but on the opposite side in some 2-cube .
All $\vec{\epsilon}_i$ are dual to the same hyperplane. Moreover, any two 1-cubes $\vec{\epsilon}, \vec{\epsilon'}$ that are dual to the same hyperplane will lie in a sequence as above beginning with $\vec{\epsilon}$ and ending with $\vec{\epsilon'}$. 

A 1-cube $\epsilon$ can be given an orientation by choosing the subcube containing the initial 0-cube of $\vec{\epsilon}$. 

\begin{rem}
    When it is not ambiguous, we can write $H_i \cdot \vec{\epsilon}_i$ instead of $Z_i \cdot \vec{\epsilon}_i$, where $H_i$ is the hyperplane that contains $Z_i$. For any hyperplane and 1-cube, $H \cdot \vec{\epsilon}$ is only defined when $\vec{\epsilon}$ is a 1-cube of a unique 2-cube with a unique midcube in $H$. Self-intersection and self-osculation will prevent uniqueness. 
    \label{rem:ref-hyperplanes-ambiguous}
\end{rem}

\begin{example}
    In $\widehat{D}_n$, an oriented 1-cube is labelled by $\tau \in PF_n(1)$. 
    Let $\sigma \in PF_n(2)$ with internal edges $E(\sigma) = \{ e_1, e_2  \}$ and $\tau = d_{e_1}(\sigma)$. Let $Z_1$ and $Z_2$ be the midcubes dual to the 1-cubes $d_{e_2}(\sigma)$ and $d_{e_1}(\sigma)$, respectively. Then $Z_1 \cdot \tau = d_{e_1} \circ r_{e_1}(\sigma)$ and $Z_2 \cdot \tau = d_{e_1} \circ r_{e_2} (\sigma) = r_{e_2}(\tau)$. See Figure~\ref{fig:reflections_by_midcubes}.

\begin{figure}[h]
    \centering
    \includegraphics[width=0.6\textwidth]{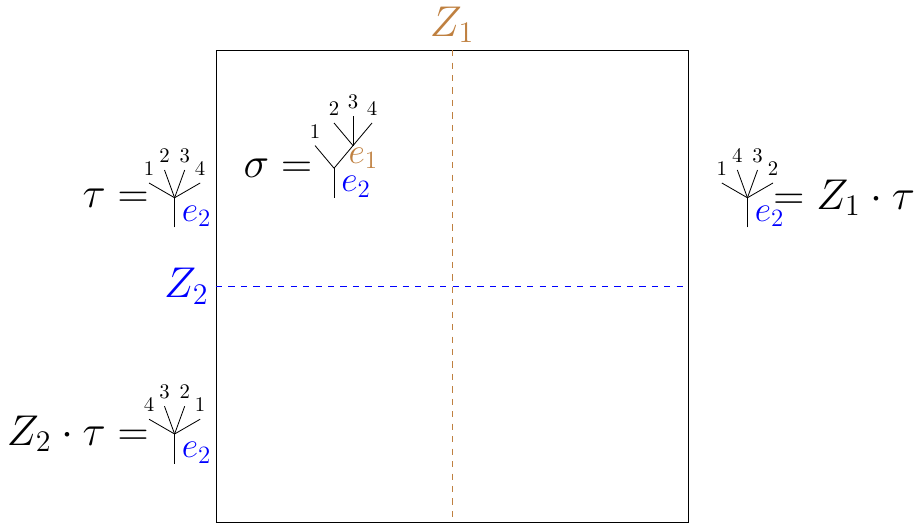} 
    \caption{Reflections $Z_i \cdot \tau$ in $\widehat{D}_3$}
    \label{fig:reflections_by_midcubes}
\end{figure}

\end{example}

The complex $\widehat{D}_n$ has only one 0-cube. All directed 1-cubes labelled by $ PF_n(1)$ are loops in $\widehat{D}_n$ and form a generating set $\{ s_A \mid A=\mathcal{O}(\tau), \tau \in PF_n(1) \}$ for the fundamental group $\pi_1(\widehat{D}_n)$. See Figure~\ref{fig:loops-Dn}.

\begin{figure}[h]
    \centering
    \includegraphics[width=0.5\textwidth]{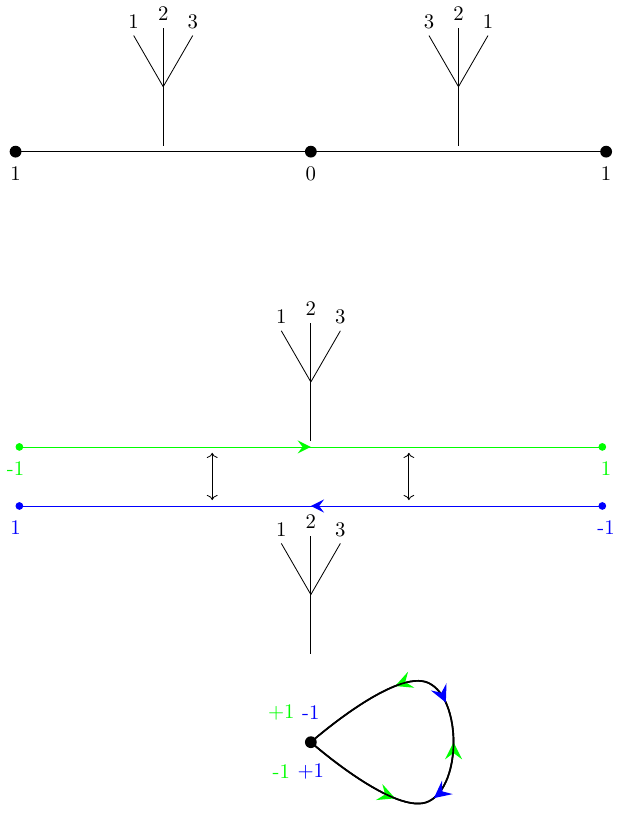}  
    \caption{As a complex made of subcubes, a 1-cube is made from gluing two 1-subcubes at a point. As a complex made of cubes, a 1-cube is made from identifying two 1-cubes in reversed orientation.}
    \label{fig:loops-Dn}
\end{figure}

\begin{rem}
    The fundamental group $\pi_1(\widehat{D}_n)$ is generated by the elements $s_A$ subject to the relations,
    $$s_As_{A^r} = 1, \hspace{1cm}
    s_As_B = s_Bs_A, \hspace{1cm}
    s_{A^r}s_{CAB} = s_{CA^rB}s_A $$ 
    where $A, B, C$ are any disjoint ordered subsets of $[n]$,
It is easy to verify this presentation once we know the 2-cubes in Figure~\ref{fig:types_of_squares} are the only kinds of 2-cube possible in $\widehat{D}_n$.
    \label{rem:fund-gp-types-squares}
\end{rem}

\begin{figure}[h]
    \centering
    \includegraphics[width=0.6\textwidth]{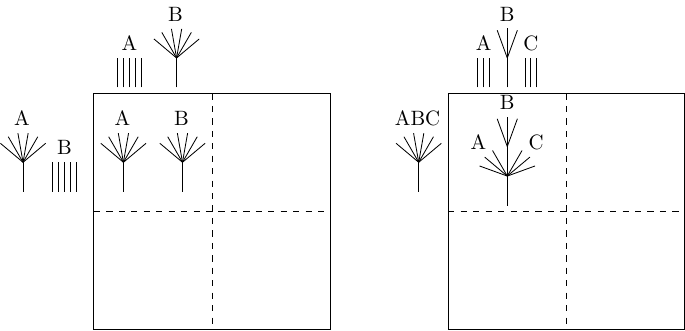}\\
    \caption{The two types of 2-subcubes in $\widehat{D}_n$}
    \label{fig:types_of_squares}
\end{figure}

In~\cite{IKLPR23}*{Theorem 11.2}, it is shown that for every $n$,  $\widehat{D}_n$ is an NPC cube complex and $\pi_1(\widehat{D}_n) \cong PvC_n$. In particular, the virtual cactus group is virtually the fundamental group of an NPC cube complex. 


\section{Finite cover for $\widehat{D}_n$ }

Let $L$ be the collection of subsets of $[n]$ that are not singletons and not empty. It has cardinality $|L| = 2^n -n-1$. For any forest $\tau$, there is a map $\mathcal{L}: E(\tau) \to L$ where $\mathcal{L}(e) \subset [n]$ is the unordered set of labels of the leaves in $\tau$ above $e$. To make notation easier, if $\tau$ has a single internal edge, $\mathcal{L}(\tau)$  is the set labels above that only edge. We call $\mathcal{L}(\tau)$ the \tit{type} of $\tau$.
 Let the forest $\tau$ label a 1-cube in $\widehat{D}_n$. The \tit{type} of the 1-cube is $\mathcal{L}(\tau)$. The following lemma allows us to define the type of a midcube and a hyperplane in $\widehat{D}_n$ by the type of its dual 1-cube. The notation $\mathcal{L}(\cdot)$ is used to denote the type of a forest, a 1-cube, a midcube or a hyperplane.

\begin{lemma}
    Two oriented $1$-cubes in $\widehat{D}_n$ have the same type if and only if they are dual to the same hyperplane. 
    \label{lemma:hyperplaneType}
\end{lemma}

\begin{proof}
    We check that all oriented 1-cubes dual to the same midcube in a 2-cube have the same type. 
    
    Let $\sigma \in PF_n(2)$ label a 2-subcube and let $E(\sigma)=\{e_1, e_2\}$. 
    Let $Z$ be the midcube dual to the oriented 1-cube labelled by $d_{e_1}(\sigma)$. Then the four oriented 1-cubes dual to $Z$ are 
    $d_{e_1}(\sigma)$, 
    $(d_{e_1}\circ~r_{e_2})(\sigma)$,    
    $(d_{e_1}\circ~r_{e_1})(\sigma)$, 
    $(d_{e_1}\circ~r_{e_2}\circ~r_{e_1})(\sigma)$. 

    There are three cases to check: $e_1, e_2$ lie in disjoint trees, $e_1$ lies above $e_2$, and $e_2$ lies above $e_1$.
    In the first case, suppose $A = \mathcal{O}(  d_{e_1}(\sigma)  )$. Then 
    $A=   \mathcal{O}(d_{e_1}\circ~r_{e_1}(\sigma))$ 
    and $A^r = \mathcal{O}(d_{e_1}\circ~r_{e_2}(\sigma)) = \mathcal{O}(d_{e_1}\circ~r_{e_2}\circ~r_{e_1}(\sigma))$. As unordered sets, these are all equal.

    For case two, let $B = \mathcal{O}(\sigma, e_1)$. Notice $B$ embeds as an interval into $\mathcal{O}(\sigma, e_2)$. Let $\mathcal{O}(\sigma, e_2)$ be the concatenation~$ABC$ of disjoint (potentially empty) ordered subsets of~$[n]$. 
    Then $\mathcal{O}(d_{e_1}(\sigma)) = ABC$, $\mathcal{O}(d_{e_1}\circ~r_{e_2}(\sigma)) = C^rB^rA^r$, $\mathcal{O}(d_{e_1}\circ~r_{e_1}(\sigma)) = AB^rC$, $\mathcal{O}(d_{e_1}\circ~r_{e_2}\circ~r_{e_1}(\sigma)) = C^rBA^r$. As unordered sets, they are equal.

    For the last case, let $B = \mathcal{O}(\sigma, e_2)$. It embeds as a interval into $\mathcal{O}(\sigma, e_1) = ABC$ where $A,B,C$ are disjoint. We can verify that  $\mathcal{O}(d_{e_1}(\sigma))=B$, $\mathcal{O}(d_{e_1}\circ~r_{e_2}(\sigma)) = B^r$, $\mathcal{O}(d_{e_1}\circ~r_{e_1}(\sigma)) = B^r$, $\mathcal{O}(d_{e_1}\circ~r_{e_2}\circ~r_{e_1}(\sigma)) = B$. As unordered sets, they are equal.

    \begin{figure}[h]
        \centering
        \includegraphics[width=0.6\textwidth]{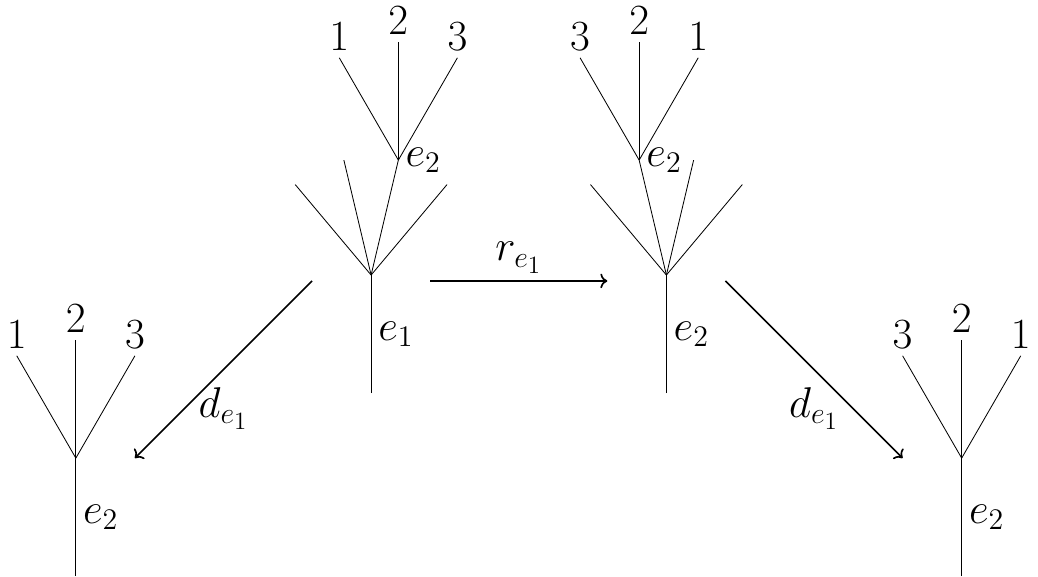}\\
        \caption{An example of the argument for case 3 when $e_2$ lies above $e_1$. We choose not draw the trees in the forests with a single internal edge.}
    \end{figure}

    To show the other direction, suppose two oriented 1-cubes labelled by $\tau_1,\tau_2\in~PF_n(1)$ are of the same type and $\tau_1 \neq \tau_2$. In other words, the ordered sets $\mathcal{O}(\tau_1)$ and $\mathcal{O}(\tau_2)$ are not equal. We can swap the order of two consecutive elements $a,b$ in $\mathcal{O}(\tau_1)$. This corresponds to inserting an internal edge $e$ to $\tau_1$, denote this new forest $\sigma\in PF_n(2)$, such that $d_e(\sigma)=\tau_1$ and $\mathcal{O}(\sigma, e) = (a,b)$ are the consecutive elements being swapped. The 1-cube labelled by $d_e~\circ~r_e(\sigma)$ is dual to the same midcube as $\tau_1$ and swaps $a,b$. If $\mathcal{O}(\tau_1)$ as only two elements, we cannot insert such an edge. Then the swapping of the elements corresponds to taking the reverse 1-cube, which is still dual to the same hyperplane. In general, $\mathcal{O}(\tau_1)$ and $\mathcal{O}(\tau_2)$ will differ be a finite number of swaps.

\end{proof}

\begin{rem}
  Given a 1-cube labelled by $\tau$, the type $A = \mathcal{L}(\tau)$ limits the possible types $B = \mathcal{L}(Z)$ of a reflecting midcube. From Figure~\ref{fig:types_of_squares}, the reflection $Z\cdot \tau$ can have no affect on $\tau$ or can reverse a subinterval of $A$. These correspond to the types being disjoint or nested. In all cases, the (possibly empty) interval $ \mathcal{L}(\tau) \cap  \mathcal{L}(Z)$ in reversed in $\mathcal{O}(\tau)$.
  We denote the resulting ordered set as $Z \cdot \mathcal{O}(\tau) := \mathcal{O}(Z\cdot \tau)$. 
  
  Suppose we have an unordered type $l\in L$ such that $l$ and $\mathcal{L}(\tau)$ are either disjoint or nested. Furthermore, suppose that $l \cap \mathcal{L}(\tau)$  is a (possibly empty) subinterval of $\mathcal{O}(\tau)$. Then there is a midcube $Z$ with type $l$ such that $Z\cdot \tau$ is defined. Such a midcube is contained in the 2-cube labelled by $\sigma \in PF_n(2)$, the forest we get by inserting a new edge $e$ into $\tau$ such that  $d_e(\sigma) = \tau$ and $\mathcal{L}(\sigma,e) = l$.
  \label{rem:reflecting_affects_order}.
\end{rem}

We are now ready to define the finite cover $M_n$. We will identify $2^{|L|}$ copies of $\widehat{D}_n$ with elements of $(\mathbb{Z}/2\mathbb{Z})^L$. We will cut and reglue along all hyperplanes. 
In particular, we will modify the gluing conditions $(\tau, (t,0)) \sim (r_e(\tau), (t,0))$ in the subcube definition of $\widehat{D}_n$. If $e$ is the internal edge of some forest $\tau$, denote $1_e = (0,\cdots,0,1,0,\cdots,0) \in (\mathbb{Z}/2\mathbb{Z})^L$ where its only non-zero entry is in the coordinate $\mathcal{L}(\tau, e)$. We define $1_{\alpha}$ in exactly the same way for all objects $\alpha$ with a type $\mathcal{L}(\alpha)$ (1-cubes, midcubes, hyperplanes).

\begin{definition}
    Take the collection of disjoint subcubes $\{ (\tau, \lbrack 0,1 \rbrack^{E(\tau)}, s)\}$ for all forests $\tau \in PF_n$ and all parameters $s \in  (\mathbb{Z}/2\mathbb{Z})^L$.  Let 
    
    $$  M_n = \bigcup_{\tau \in PF_n} \{\tau \} \times [0,1]^{E(\tau)} \times (\mathbb{Z}/2\mathbb{Z})^L  /\sim$$
where we impose the three gluing conditions for all $\tau, s$:
\begin{itemize}
    \item $ (\tau , (t,1), s)  \sim (d_e(\tau), t , s)$
    \item $ (\tau , t, s)  \sim (\tau', t , s)$ if $\tau'$ is just the permutation of trees in $\tau$
    \item  $(\tau , (t,0), s)  \sim (r_e(\tau), (t, 0), s + 1_e)$
\end{itemize}
\label{def:cover}
\end{definition}

\begin{lemma}
    The cube complex $M_n$ is a finite cover of $\widehat{D}_n$. The covering map is defined as
    $$P: M_n \to \widehat{D}_n$$
    $$(\tau,t,s) \mapsto (\tau, t)$$
    The degree of this covering is $2^{2^n -n-1}$.
    \label{lemma:covering-map}
\end{lemma}
\begin{proof}
    The cubes in $M_n$ are made from grouping subcubes together such that the intersection of two groups of subcubes satifies the definition of a cube complex.

    Pick a $k$-subcube labelled by $(\tau, s)$. Consider the collection of $k$-subcubes that are glued to an inward face of $(\tau, s)$. Consider the $k$-subcubes that are glued to an inward face of these $k$-subcubes, and so forth. The total collection of these subcubes is $\{ (r_{e_1}\cdots r_{e_m}) (\tau), s+1_{e_1} +1_{e_2}+\dots+1_{e_m} | e_i \in E(\tau)   \}$ where the $e_i$'s are not necessarily distinct and $m\in \mathbb{N}$. Let the union of these subcubes be $C$. 
    Then $C$ is a cube and the groups of inward faces are the midcubes of $C$. 
    This is because reflections about the inward faces commute $Z_1\cdot Z_2 \cdot(\tau,s) = (Z_1\cdot Z_2 \cdot\tau, s + 1_{Z_1} + 1_{Z_2})=(Z_2\cdot Z_1 \cdot \tau, s + 1_{Z_2} + 1_{Z_1})  =Z_2\cdot Z_1 \cdot(\tau,s)$ and $Z_1\cdot Z_1 \cdot(\tau,s)=(Z_1\cdot Z_1 \cdot \tau, s+1_{Z_1}+1_{Z_1}) = (\tau,s)$. The isometry from $C$ to the cube $[-1,1]^{E(\tau)}$ can be defined by sending  $(\tau,s)$ to the positive subcube $[0,1]^{E(\tau)}$ and sending the inward faces of $(\tau,s)$ to the inward faces of $[0,1]^{E(\tau)}$.

    To see that $P$ is a covering map, notice that the map is cubical and induces an isomorphism between the link of a vertex in $M_n$ and link of a vertex in $\widehat{D}_n$. 

\end{proof}


\begin{rem}
The fundamental group of $M_n$ is the kernel of the map 
$$PvC_n \to (\mathbb{Z}/2\mathbb{Z})^L$$
$$w=s_{A_1} s_{A_2} \cdots s_{A_m} \mapsto 1_{A_1} + \dots+1_{A_m}$$

\end{rem}

We define the type of 1-cubes, midcubes, and hyperplanes in $M_n$ to be the type of the corresponding features in the base space $\widehat{D}_n$. Unlike $\widehat{D}_n$, there may be different hyperplanes in $M_n$ with the same type (Figure~\ref{fig:M_n-hyperplane_types}). To keep the notation simple, we will continue to use $\mathcal{L}$ to denote the type of a feature in $M_n$. The reflections about midcubes in $M_n$ satisfy  $Z \cdot (\tau, s) = (P(Z)\cdot \tau, s + 1_{Z})$ where $P$ is the covering map from Lemma~\ref{lemma:covering-map}.

\begin{figure}[h]
    \centering
    \includegraphics[width=0.6\textwidth]{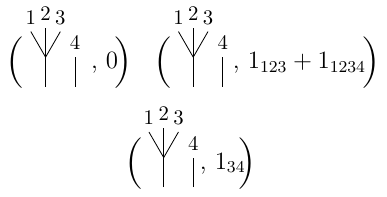}\\
    \caption{In $M_n$, the two forests in the first row label 1-cubes dual to the same hyperplane. The forest in the second row labels a 1-cube that is dual to a different hyperplane. We leave the justification as an exercise for the reader.}
    \label{fig:M_n-hyperplane_types}
\end{figure}


\section{$M_n$ is a Special Cube Complex}

Consider a 1-cube $\epsilon$ and its dual hyperplane in $\widehat{D}_n$. Since $\widehat{D}_n$ has a single 0-cube, the endpoints of $\epsilon$ are identified. The oriented 1-cube $\vec{\epsilon}$ and its reverse oriented 1-cube share an initial 0-cube. Thus the hyperplane self-osculates. Although $\widehat{D}_n$ contains unwanted hyperplane pathologies, we will show that the finite cover $M_n$ does not.

\begin{lemma}
    $\widehat{D}_n$ has no self-intersecting hyperplanes.
    \label{lemma:dn_selfintersection}
\end{lemma}

\begin{proof}
     Let forests $\tau_1,\tau_2$ label two dual 1-subcubes to the hyperplane $H$. By Lemma~\ref{lemma:hyperplaneType}, they have the same type. They also have the same endpoint, the only 0-cube in $\widehat{D}_n$. 
     Suppose they form a corner of a 2-subcube. Since the link is simplicial, $\tau_1 \neq \tau_2$. The types are not disjoint so they can form a corner of a 2-subcube only if the ordered sets $\mathcal{O}(\tau_1)$ and $\mathcal{O}(\tau_2)$ are nested as an interval (Remark~\ref{rem:fund-gp-types-squares}). 
     Thus $\mathcal{O}(\tau_1) = \mathcal{O}(\tau_1)$ which implies~$\tau_1=\tau_2$. Contradiction.
\end{proof}

In a general setting, suppose an NPC cube complex $X$ has embedded hyperplanes, each given a unique type. Suppose hyperplanes of a finite cover $\widetilde{X}$ receive types from the base space $X$. Then any two hyperplanes in $\widetilde{X}$ of the same type cannot intersect because their image under the covering map is an embedded hyperplane. More specifically, any two dual 1-cube of the two hyperplanes in $\widetilde{X}$ that form a corner of a 2-cube will map to two distinct dual 1-cube of a single embedded hyperplane in $X$ and will form a corner of a 2-cube, which is a contradicton. The following corollary is a result of this general discussion.

\begin{corollary}
    Hyperplanes of the same type do not intersect in $M_n$. In particular, $M_n$ has no self-intersecting hyperplanes.
    \label{cor:selfintersection}
\end{corollary}


\begin{lemma}
    Hyperplanes in $M_n$ are two-sided.
    \label{lemma:two-sided}
\end{lemma}
\begin{proof}
    Let $(\tau_1, s_1)$ label an oriented 1-cube in $M_n$ that is dual to the hyperplane $H$ and let $e$ be the only internal edge of $\tau_1$. Consider a sequence of midcubes $Z_1,\dots,Z_m$ in $M_n$ such that the sequence of reflections $(\tau_{i}, s_{i}) = Z_{i-1} \cdot (\tau_{i-1}, s_{i-1})$ is defined and each $Z_i$ is not dual to the 1-cube $(\tau_i,s_i)$. Notice that $(\tau_{i-1}, s_{i-1})$ is the 1-cube opposite of $(\tau_{i}, s_{i})$ in some 2-cube and they are oriented in the same direction. 

    Let $l_i\in L$ be the type of $Z_{i}$. Notice that the reflection $Z_{i-1} \cdot (\tau_{i-1}, s_{i-1})$ modifies the parameter by adding 1 to the index corresponding to the type of the midcube. More precisely,  $s_i = s_{i-1} + 1_{l_i}$. By Corollary~\ref{cor:selfintersection}, we have $l_i\neq~\mathcal{L}(\tau_1)$, thus the coordinate  $\mathcal{L}(\tau_1)$ is the same in all parameters $s_i$. The reverse oriented 1-cube $(r_e(\tau_1), s_1 +1_e)$ cannot appear in the sequence. Therefore if we extend the orientation of $(\tau_1,s_1)$ to all dual 1-cubes of $H$, the orientation of each 1-cube will be unique.

\end{proof}

\begin{rem}
Notice that the proof of Lemma~\ref{lemma:two-sided} did not use anything specific to ~$\widehat{D}_n$. Take any NPC cube complex ~$X$ with embedded hyperplanes and a finite cover ~$\widetilde{X}$ constructed by cutting and regluing along the hyperplanes as we did in Definition~\ref{def:cover}. The proof in Lemma~\ref{lemma:two-sided} shows that the hyperplanes in ~$\widetilde{X}$ are two-sided.
\end{rem} 


\begin{lemma}
    Let $\tau_1,\dots,\tau_m$ label a sequence of oriented 1-cubes in $\widehat{D}_n$ such that each subsequent 1-cube comes from a reflection $\tau_i=Z_{i-1}\cdot\tau_{i-1}$ for some midcube $Z_{i-1}$. Let $H_{i}$ be the hyperplane containing~$Z_i$. If each hyperplane appears an even number of times in the bag $\{H_1,\dots,H_{m-1}\}$, then $\tau_1 = \tau_m$.
    \label{lemma:self-osc-helper}
\end{lemma}
\begin{proof}
    We know $l=\mathcal{L}(\tau_1) = \mathcal{L}(\tau_m)$ from Lemma~\ref{lemma:hyperplaneType}. To prove $\tau_1=\tau_m$, we must show $\mathcal{O}(\tau_1) = \mathcal{O}(\tau_m)$. In other words, for all $a,b\in l$, $a<b$ in $\mathcal{O}(\tau_1)$ if and only if $a<b$ in $\mathcal{O}(\tau_m)$. 

    Let $l_1,\dots,l_{m-1} $ be the types of $Z_1, \dots, Z_{m-1}$, respectively. The reflection $Z_i \cdot \tau_i$ reverses the (possibly empty) interval $l_i \cap l$ contained in $\mathcal{O}(\tau_i)$ (Remark~\ref{rem:reflecting_affects_order}). For any $a,b \in l$, notice that the number of types $l_i$ that contain both $a,b$ is even. Since each corresponding reflection $Z_i \cdot \tau_i$ reverses the order of $a$ and $b$, the composition of all the reflections results in an even number of reversals of $a,b$. 
    Thus $\mathcal{O}(\tau_m) = Z_{m-1} \cdots  Z_1 \cdot \mathcal{O}(\tau_1)~=~\mathcal{O}(\tau_1)$.

\end{proof}

\begin{lemma}
    Hyperplanes in $M_n$ do not self-osculate.
    \label{lemma:self-osc}
\end{lemma}
\begin{proof}
    Suppose  $(\tau,s),(\tau',s')$ label oriented 1-cubes dual to a hyperplane $H$ in $M_n$. Suppose they originate from the same 0-cube. Then the parameters must be equal, $s=s'$. Furthermore, there is a sequence of reflections via midcubes $Z_1,\dots,Z_m$ such that $(\tau',s') = Z_m \cdots Z_1 \cdot (\tau,s) = (P(Z_m) \cdots P(Z_1) \cdot \tau , s + 1_{Z_1} + \cdots + 1_{Z_m})$ where $P$ is the covering map. Since $s' = s$, each $1_{Z_i}$ must appear an even number of times in the sum. Consequently, each hyperplane $H_i \in \widehat{D}_n$ containing $P(Z_i)$ must appear an even number of times. By Lemma~\ref{lemma:self-osc-helper}, $P(Z_m) \cdots P(Z_1) \cdot \tau = \tau$. Therefore $(\tau,s) = (\tau',s')$.

\end{proof}

\begin{definition}
    Suppose we have a hyperplane $H$ and a 1-cube $\vec{\epsilon}$ in some cube complex $X$.
    Consider the collection $\{ Z \in H |  Z \cdot \vec{\epsilon}$ is defined $  \}$ of all midcubes $Z$ in $H$ such that $Z \cdot \vec{\epsilon}$ is defined. We say $H\cdot \vec{\epsilon}$ is \tit{valid} if the set is non-empty and \tit{invalid} otherwise.

    In our case, if $X$ is $\widehat{D}_n$ or $M_n$, just knowing the types $\mathcal{L}(H)$ and $\mathcal{L}(\epsilon)$ is not enough to determine if $H \cdot \vec{\epsilon}$ is valid. However, if the types are not nested and not disjoint, $H\cdot \vec{\epsilon}$ must be invalid. We say the reflection $H\cdot \vec{\epsilon}$ is \tit{strongly invalid}. We also say the types are \tit{strongly invalid}. See Figure~\ref{fig:strongly-invalid} for an example.
    \end{definition}

    \begin{figure}[h]
      \centering
      \includegraphics[width=0.3\textwidth]{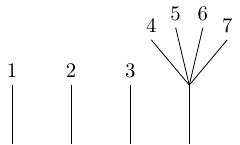}
      \caption{For the 1-cube labelled by the forest above, here are some invalid and strongly invalid types. 
      Invalid types: $\{4,6\}$, $\{4,5,7\}$, $\{1,5,7\}$.
      Strongly invalid types: $\{1,5,7\}$, $\{2,3,4,5\}$, $\{1,2,3,4,5,6\}$.
      }
      \label{fig:strongly-invalid}
  \end{figure}


\begin{lemma}
    Hyperplanes in $M_n$ do not inter-osculate.
    \label{lemma:inter-osc}
\end{lemma}
\begin{proof}
    We begin with two intersecting hyperplanes $H_1,H_2$ in $M_n$. 
    There exists a 2-subcube labelled by $(\sigma, s)$ with internal edges $E(\sigma) = \{e_1, e_2\}$ such that $H_1$ and $H_2$ are dual to the oriented 1-cube $(d_{e_2}(\sigma), s)$ and $(d_{e_1}(\sigma), s)$. 

    Case 1: Types  $\mathcal{L}(\sigma, e_i)$ are disjoint.

    Take arbitrary dual oriented 1-cubes $(\tau, s'), (\rho,s')$ of $H_1, H_2$, respectively, that originate from the same 0-cube. By Lemma~\ref{lemma:hyperplaneType}, $\rho$ and $\tau$ will still have disjoint types. In the base space $\widehat{D}_n$,  $\rho$ and $\tau$ label 1-cubes of a corner of a 2-subcube (Remark ~\ref{rem:fund-gp-types-squares}). The lift of this 2-subcube where the parameter is $s'$ demonstrates that  $(\rho,s'),(\tau, s')$ also form a corner of a 2-subcube. So it is impossible for these hyperplanes to inter-osculate. This kind of inter-osculation is impossible for $\widehat{D}_n$ as well.

    Case 2: Types $\mathcal{L}(\sigma, e_i)$ are nested. Suppose the internal edge $e_1$ is below $e_2$ in $\sigma$ as in Figure~\ref{fig:io-case2-im1}. 

    \begin{figure}[h]
        \centering
        \includegraphics[width=0.4\textwidth]{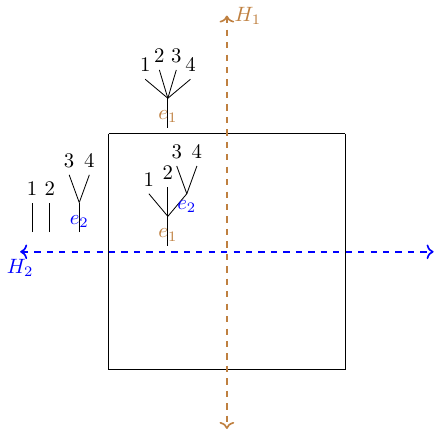}
        \caption{}
        \label{fig:io-case2-im1}
    \end{figure}

    Let   $(\tau_1, s_1) = (d_{e_2}(\sigma), s)$ 
    and $(\rho_1, t_1) = (d_{e_1} (\sigma), s)$. Once again, take arbitrary $(\tau, s'), (\rho,s')$ dual to $H_1, H_2$, respectively, such that they originate from the same 0-cube. We will prove that $\mathcal{O}(\rho)$ embeds as an interval into $\mathcal{O}(\tau)$. Then by Remark~\ref{rem:fund-gp-types-squares}, the oriented edges $\rho, \tau$ form a corner of a 2-subcube. The remaining argument is the same as case 1.
    
    There is some sequence of reflections taking $(\rho_1,s)$ to $(\rho',s')$. By Remark~\ref{rem:reflecting_affects_order}, none of the reflecting midcubes can have a type that is strongly invalid with $\mathcal{L}(H_2)$. The types of these midcubes modify the parameter, thus if $l\in L$ is strongly invalid with $\mathcal{L}(H_2)$, the corresponding coordinate in $s$ and $s'$ must be the same. 

    There is also a sequence of midcubes $Z_1,\dots,Z_{m-1}$ in $M_n$ such that the sequence of reflections $(\tau_{i}, s_{i}) = Z_{i-1} \cdot (\tau_{i-1}, s_{i-1})$ is defined and $(\tau_m, s_m) = (\tau, s')$. Let $l_i = \mathcal{L}(Z_i)$ be the type of the midcubes. Each type $l_i$ that is strongly invalid with the type $\mathcal{H_2}$ must appear an even number of times by the previous argument.
    We first show $\mathcal{L}(\rho) = \mathcal{L}(H_2)$ is an interval in $\mathcal{O}(\tau)$, then show $\mathcal{O}(\rho)$ is embedded in $\mathcal{O}(\tau)$.

    \begin{figure}[h]
        \centering
        \includegraphics[width=0.8\textwidth]{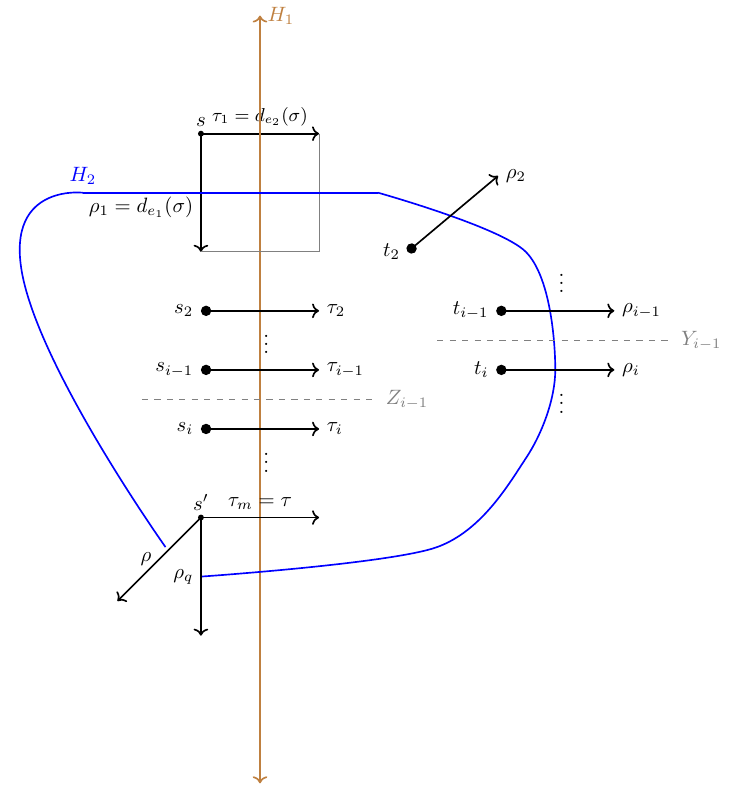}
        \caption{}
        \label{fig:io-ladder}
    \end{figure}



Let $x \in \mathcal{L}(H_1) - \mathcal{L}(H_2)$. 
For each forest $(\tau_i,s_i)$, $x$ partitions $\mathcal{L}(H_2)$ into two sets 
$A_i = \{ y \in \mathcal{L}(H_2) \mid y < x $ in the ordering $ \mathcal{O}(\tau_i) \}$ 
and 
$B_i = \{ y \in \mathcal{L}(H_2) \mid y > x $ in the ordering~$ \mathcal{O}(\tau_i) \}$. 
Suppose initially $A_1 = \mathcal{L}(H_2)$ and $B_1=\varnothing$ and $a,b\in A_1$.
If the type $\mathcal{L}(Z_1)$ contains $\{a,x\}$ but not $b$ or  $\{b,x\}$ but not $a$, then the partition given by $\tau_2$ places $a,b$ into disjoint sets. Also note that the types $\mathcal{L}(Z_1)$ and $\mathcal{L}(H_2)$ are strongly invalid. Consider all $Z_i$ such that $\mathcal{L}(Z_i)$ contains $\{a,x\}$ but not $b$ or  $\{b,x\}$ but not $a$. The corresponding $\tau_{i+1}$ will partition $a,b$ into disjoint subsets if they where previously in the same subset and vice versa. Each strongly invalid reflection occurs an even number of times so $a,b$ will lie in the same set in the final partition. Thus the unordered set $ \mathcal{L}(\rho)$ forms an interval in $\mathcal{O}(\tau)$. It remains to show that the ordered set $\mathcal{O}(\rho)$ is embedded in $\mathcal{O}(\tau)$.

    Define a new sequence from the types $l_i$ as follows. 
\[ l_i' =   \left\{
\begin{array}{ll}
      l_i & \text{if } l_i, \mathcal{L}(H_2)  \text{ are not strongly invalid} \\
      l_i \cap \mathcal{L}(H_2) & \text{if } l_i, \mathcal{L}(H_2) \text{ are strongly invalid} \\
\end{array} 
\right. \]

Notice that $l_i'$ cannot be empty. We claim that there is a strictly increasing function $\mathcal{I}:[q] \to [m]$ such that $\mathcal{I}(q)=m$ and a sequence $(\rho_1, t_1) = (d_{e_1} (\sigma), s), (\rho_2, t_2), \dots, (\rho_q, t_q) = (\rho, s')$ of 1-cubes dual to $H_2$ such that the ordered set $\mathcal{O}(\rho_i)$ embeds into $\mathcal{O}(\tau_{\mathcal{I}(i)})$ (not necessarily as an interval). The lemma immediately follows from this claim.

Assume the inductive hypothesis: for some $i$, $\mathcal{O}(\rho_i)$ embeds into $\mathcal{O}(\tau_{\mathcal{I}(i)})$. 

\begin{figure}[h]
    \centering
    \includegraphics[width=0.8\textwidth]{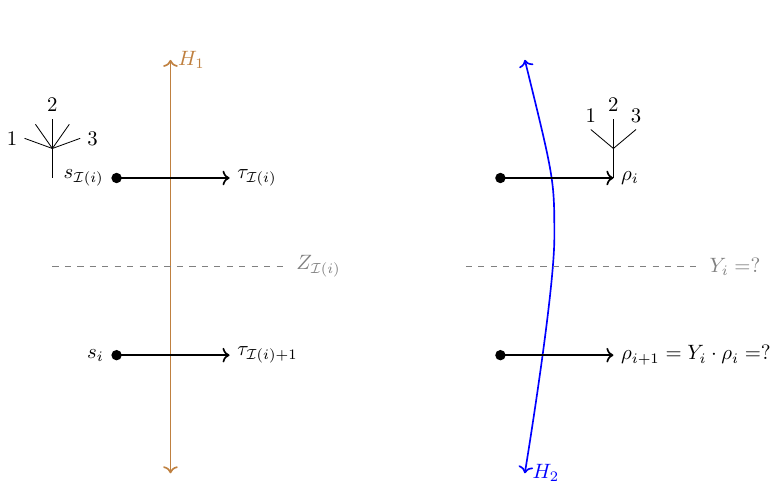}
    \caption{}
    \label{fig:io-next-rho}
\end{figure}

The reflection by the midcube $Z_{\mathcal{I}(i)}$ reverses the interval $l_{\mathcal{I}(i)}$ in the ordered set $\mathcal{O}(\tau_{\mathcal{I}(i)})$. By our inductive hypothesis, $l_{\mathcal{I}(i)}'$ is either an interval (possibly a singleton) in, or is disjoint from, $\mathcal{O}(\rho_i)$. 

If $l_{\mathcal{I}(i)}'$ is not a singleton, extend the definition $\mathcal{I}(i+1) = \mathcal{I}(i) + 1$. By Remark~\ref{rem:reflecting_affects_order}, there exists a midcube $Y_i$ of type $\mathcal{L}(Y_i) = l_{\mathcal{I}(i)}'$ such that $Y_i \cdot (\rho_i, s_i)$ is defined. We let $(\rho_{i+1}, s_{i+1})$ be the result of this reflection.

If $l_{\mathcal{I}(i)}'$ is a singleton, no midcube will have this type, so consider instead the next $l_{\mathcal{I}(i)+1}'$. Suppose $j$ is the smallest natural number such that $l_{\mathcal{I}(i)+j}'$ is not a singleton. Extend the definition $\mathcal{I}(i+1) = \mathcal{I}(i) + j+1$, pick a midcube $Y_i$ of type $\mathcal{L}(Y_i) = l_{\mathcal{I}(i)+j}'$ such that $  Y_i \cdot (\rho_i, s_i)$ is defined and $(\rho_{i+1}, s_{i+1})$ is the result of the reflection.

We check that the inductive property holds. The sequence of reflections $Z_{\mathcal{I}(i) + j} \cdots Z_{\mathcal{I}(i)} \cdot \tau_{\mathcal{I}(i)} = \tau_{\mathcal{I}(i) + j+1}$ changes the ordered set $\mathcal{O}(\tau_{\mathcal{I}(i)})$. 
Among all these midcubes, the reflection about $Z_{\mathcal{I}(i) + j}$ is the only one that can affect the suborder $\mathcal{O}(\rho_i)$ (the types of all other midcubes intersect with $\mathcal{O}(\rho_i)$ to give a singleton).
The reflection about $Z_{\mathcal{I}(i) + j}$  reverses $l'_{\mathcal{I}(i)+j} \cap \mathcal{L}(H_2)$ in this suborder. Since $\mathcal{L}(Y_i) = l'_{\mathcal{I}(i)+j}$, this is precisely what the reflection $Y_i \cdot (\rho_i,s_i)$ does to $\mathcal{O}(\rho_i)$. Thus the inductive property holds.

If such a value $j$ does not exist, redefine $\mathcal{I}(i) =m$, terminate the process and let $q = i$. We have $\displaystyle t_q = s + \sum_{i=1}^{q-1} 1_{Y_i}$  and  $ \displaystyle s' = s + \sum_{i=1}^{m-1} 1_{Z_i} $. Notice 
$$t_q - s' =  \displaystyle \sum_{i \in \mathcal{Y}}  1_{l'_i} - \sum_{i \in \mathcal{Z}} 1_{l_i}$$
where 
$\mathcal{Z} = \{i \mid l_i \text{ and } \mathcal{L}(H_2)  \textrm{ are strongly invalid}  \}$ and
$\mathcal{Y} = \{i \mid l_i \text{ and } \mathcal{L}(H_2)$   are strongly invalid and $l_{i}'$ has at least 2 elements$ \}$. Since each strongly invalid type appears an even number of times, $t_q - s' = 0 \in (\Z/2\Z)^L$. The two 1-cubes $(\rho_q, s'), (\rho, s')$ originate from the same 0-cube and are both dual to $H_2$. By Lemma~\ref{lemma:self-osc} and Corollary~\ref{cor:selfintersection}, we have equality~$\rho_q =\rho$.

\end{proof}

The main results Theorem~\ref{thm:Mn-special} and Corollary~\ref{cor:vcn-vs} follow from Lemma~\ref{lemma:dn_selfintersection}, Lemma~\ref{lemma:two-sided}, Lemma~\ref{lemma:self-osc} and Lemma~\ref{lemma:inter-osc}. 

\smallskip
\textbf{Acknowledgements:} I would like to thank my supervisors Joel Kamnitzer and Piotr Przytycki for proposing this project, for their patients and for guiding me over the past 2 years. I would like to thank Piotr for proofreading countless versions of this article and providing meaningful feedback every time. Lastly, I would like to thank my family for always being understanding and for always being by my side in the face of adversities.

\bibliography{vcn_vs}

\end{document}